\documentclass{shortnotes}
\usepackage{projectcommands} 
\title{$p$-adic Hodge theory of unipotent completions revisited}
\author{Ruth Wild}
\date{\today}
\abstracttext{Let $K/\qp$ be a finite extension, and let $X/K$ be an algebraic variety. We present a new proof that the unipotent completion of the \'etale fundamental group of $X_{\bark}$ is a pro-de Rham representation, and show that it is pro-semi-stable whenever $X$ has semi-stable \'etale cohomology. We also reprove, using syntomic methods, that it is pro-crystalline at integral base points whenever $X$ admits a smooth proper formal model. All three proofs use the bar construction and are inspired by Wojtkowiak's approach to unipotent fundamental groups.}

\begin{document}
\maketitlepage
\section{Introduction}
Let $K/\qp$ be a finite extension with algebraic closure $\bark$ and Galois group $G_K$. 
Let $X$ be an algebraic variety over $K$.
In $p$-adic Hodge theory, various properties of the \'etale cohomology groups
$
\hh^i(X_{\bark},\qp)
$
as $G_K$-representations have been established:\footnote{We do not attempt to give an exhaustive list of attributions for these deep theorems, as this would be longer than the current body of our paper. Also, we restrict the attributions to results proved in the algebraic setting. Of course, there is also an extensive list of results for more general rigid-analytic varieties.}
They are always de Rham representations, e.g.\ \cite{beilinson,deglise-niziol,kisin-pss}.
If $X$ arises as the generic fiber of a proper scheme $\calx$ over $\ok$ with semi-stable reduction, then the $\hh^i(X_{\bark},\qp)$ are semi-stable Galois representations, e.g.\ \cite{tsuji,almost-faltings,niziol-k-theory}.
Moreover, if $\calx$ has good reduction, the $\hh^i(X_{\bark},\qp)$ are even crystalline; see, e.g., \cite{faltings-crystalline} for a proof that does not deduce this from semi-stability.
\par
A natural question is how these properties extend to the \emph{pro-unipotent completion} $\piuni(X_{\bark},x)$ of the geometric \'etale fundamental group of $X$.
This is a pro-unipotent group over $\qp$,  together with a Galois action induced by a $K$-rational point $x \in X(K)$.
Its abelianization is constant, independent of $x$ and satisfies
\[
\piuni(X_{\bark},x)^{\mathrm{ab}}(\qp)
=
\hh^1(X_{\bark},\qp)^\vee;
\]
the right-hand side is indeed a de Rham representation.
To make sense of the question in general, we make the following definition:
\begin{defn}
  \label{defn:de-rham-start}
Write $\piuni(X_{\bark},x) = \spec(A_{\infty,x})$, for some uniquely determined commutative Hopf-algebra $A_{\infty,x}$.
The $G_K$-action on $\piuni(X_{\bark},x)$ determines a $G_K$-action on $A_{\infty,x}$, such that $A_{\infty,x}$ lives naturally in $\ind(\repqp(G_K))$.
We say that $\piuni(X_{\bark},x)$ is \emph{pro-de Rham} if $A_{\infty,x}$ lies in the full subcategory $\ind(\repdr(G_K))$ of ind-de Rham representations.

\end{defn}
\begin{rem*}
In the literature, one usually defines pro-de Rhamness of $\piuni(X_{\bark},x)$ using the notion of being associated by a comparison isomorphism to the unipotent de Rham fundamental group. 
Even if our definition appears to be weaker, once pro-de Rhamness of $\piuni(X_{\bark},x)$ is established, an isomorphism 
\[
  \piuni(X_{\bark},x)
  \tensor_{\qp}
  \bdr
  \isomorphism
  \pi_1^{\mathrm{dR,uni}}(X,x)
  \tensor_K
  \bdr
\]
witnessing that $\piuni(X_{\bark},x)$ is associated to its Hodge-theoretic counterpart can be established using the notion of universal unipotent objects; see, e.g. \cite{hadian,betts-constancy}.
The only additional input from $p$-adic Hodge theory needed is that $\hh^1(X_{\bark},\qp)$ and $\hh^1_{\mathrm{dR}}(X/K)$ are associated.
\end{rem*}
The result of our paper can be stated as follows:
\begin{theorem}
  \label{thm:intro-main}
  Let $X$ be an algebraic variety over $K$, and $x\in X(K)$ a $K$-rational point.
  \begin{enumerate}
  \item
  It always holds that $\piuni(X_{\bark},x)$ is pro-de Rham.
  \item
  If all $\hh^i(X_{\bark},\qp)$ are semi-stable Galois representations, then $\piuni(X_{\bark},x)$ is pro-semi-stable.
  \item
  If $X$ arises as the generic fiber of a smooth proper formal scheme $\calx/\spf(\ok)$ and $x \in \calx(\ok)$, then $\piuni(X_{\bark},x)$ is pro-crystalline.
  \end{enumerate}
\end{theorem}
\par
As with the abelian comparison, the non-abelian statement for de Rhamness has been known for a long time; see e.g.\ \cite{vologodsky,deglise-niziol,betts-constancy}.
The statement for crystallinity has been established by \cite{olsson}.
To the best of our knowledge, semi-stability has so far only been established in the case of curves \cite{aik}.
Our approach to de Rhamness and semi-stability begins similarly as in \cite{deglise-niziol}, namely by lifting $R\Gamma(X_{\bark},\qp)$ to $\dd[b](\repdr(G_K))$.
This is not automatic: the induced functor $\dd[b](\repdr(G_K)) \to \dd[b](\repqp(G_K))$ is not fully faithful, and so containment cannot be shown on cohomology groups.
But a lift can be obtained once the \'etale--de Rham comparison has been formulated at the level of derived categories.
After this, we take a different route.
Their approach uses the motivic realization of $\piuni(X_{\bark},x)$ constructed by Cushman \cite{cushman} to bootstrap the cohomological comparison for de Rhamness from the level of derived categories to the unipotent fundamental group.
One can avoid using the language of motives by directly realizing $A_{\infty,x}$ from $R\Gamma(X_{\bark},\qp)$ through the bar construction.
This possibility was first observed by Wojtkowiak \cite{wojtkowiak}; see also \cite{olsson-bar}. In our context, we use the precise formulation for \'etale sheaves by Dan-Cohen--Schlank.
Ind-de Rhamness of $A_{\infty,x}$ can then be deduced by a formal argument from the fact that the functor $\dd[b](\repdr(G_K)) \to \dd[b](\repqp(G_K))$ commutes with bar constructions (after passing to $\ind$-categories).
\par
To obtain ind-semi-stability of $A_{\infty,x}$, we show that, unlike in the case of the inclusion of de Rham representations into all representations, the inclusion $\repst(G_K) \hookrightarrow \repdr(G_K)$ extends to a fully faithful functor on bounded derived categories; ind-semi-stability of $A_{\infty,x}$ is then again purely formal.
\par
For ind-crystallinity of $A_{\infty,x}$, however, we cannot use the same argument, since $\repcris(G_K) \sse \repst(G_K)$ is again not closed under extensions.
To get around this, we do not attempt to directly lift $R\Gamma(X_{\bark},\qp)$ to $\dd[b](\repcris(G_K))$.
Instead, we use that $R\Gamma(X_{\bark},\qp)$ can be realized from syntomic cohomology, that the (rationalized) realization functor is symmetric monoidal, and that its essential image is given by $\dd[b](\repcris(G_K))$.
This allows us to give a purely formal argument once again.
Unlike in the other two cases, this relies heavily on constructions in $p$-adic Hodge theory that are very recent.
However, we can largely treat these as black boxes.
\par
Finally, let us note that we also state and prove versions of \cref{thm:intro-main} for path torsors in all three cases.
\subsection*{Acknowledgements}
I would like to thank my PhD advisor Jakob Stix for his support and for the many clarifying and insightful discussions during the preparation of this paper.
I would also like to thank Lucas Mann and Alexander Petrov; discussions with both led to simplifications of the arguments presented.
Finally, I am grateful to Alexander Betts, Magnus Carlson and Benjamin Steklov for their interest in this work, many helpful questions and conversations about unipotent completions and their different presentations.
This project was funded by Deutsche Forschungsgemeinschaft (DFG, German Research Foundation) through the Collaborative Research Centre TRR 326 Geometry and Arithmetic of Uniformized Structures, Project ID 444845124.
\par
\textsc{AI-Disclosure:}
We used the cloud version of ChatGPT to locate various references, in particular those on $\infty$-categories.
Moreover, we used Codex with Astra 6 to check a first version of this script for mathematical errors, which lead to minor corrections; Codex was also used for proofreading and a grammatical review.

\section{Categorical preliminaries}

\begin{numtext}
  \label{text:bar-const}
  The main player of this paper is the bar construction in the context of symmetric monoidal (stable) $\infty$-categories, which we briefly recall:
  Let $(\ccat,\tensor,\idob)$ be a symmetric monoidal $\infty$-category, and $A \in \calg(\ccat)$ a commutative algebra object.
  Then one can define a category $\modcat_{A}(\ccat)$ of $A$-modules in $\ccat$ \cite[Defn.4.5.1.1]{higher-algebra}.
  If $\ccat$ admits geometric realizations of simplicial objects and $\tensor$ is exact in each variable, $\modcat_{A}(\ccat)$ can again be endowed with a symmetric monoidal structure \cite[Thm.4.5.2.1]{higher-algebra}.
  Informally, for two $A$-modules $M,N$, their tensor product is given as 
  \[
    M \tensor_A N 
    =
    \colim_{\Delta^\op}
    (M \tensor A^{\tensor n} \tensor N).
  \]
  If $\ccat$ is stable, $\modcat_A(\ccat)$ is stable too \cite[Cor.4.5.1.6, Thm.7.1.1.4]{higher-algebra}.
  Moreover, if $ x \mc A \to \idob$ is an augmentation, then 
  \[B(A,x) \defo \idob_x \tensor_A {}_x\idob\] is naturally an augmented coalgebra in $\ccat$ \cite[Thm.5.2.2.17, Rem.5.2.2.19]{higher-algebra}.
  Here, we use the notation $\idob_x$ (resp. ${}_x\idob$) to indicate that we view $\idob$ as a right- (resp. left-) $A$-module via the augmentation $x$.
  If $y \mc A \to \idob$ is a second augmentation, then 
  \[
  B(A,x,y) \defo \idob_x \tensor_A {}_y\idob
  \]
  is naturally a $B(A,x)$-$B(A,y)$-bicomodule \cite{dc-comodule}.
\end{numtext}

\begin{lem}
  \label{lem:induced-on-modules}
Let $\ccat, \dcat$ be presentable, stable, symmetric monoidal $\infty$-categories. 
Let $F\mc \ccat \to \dcat$ be a symmetric monoidal, colimit-preserving functor.
Let $A \in \calg(\dcat)$. 
Then any lift $\wtilde{A} \in \calg(\ccat)$ of $A$ induces a symmetric monoidal lift $\wtilde{F}\mc \modcat_{\wtilde{A}}(\ccat) \to \modcat_A(\dcat)$ of $F$.
\end{lem}
\begin{proof}
This follows, e.g., from \cite[Cor.4.8.5.21]{higher-algebra}.
\end{proof}

\begin{rem}
  \label{rem:lift-essential}
  In particular, if $x,y\mc A \to \idob_{\dcat}$ are two augmentations in $\dcat$, lifting to augmentations $\wtilde{x},\wtilde{y}\mc \wtilde{A} \to \idob_{\ccat}$ in $\ccat$, then 
  \[
    \wtilde{F}(B(\wtilde{A},\wtilde{x},\wtilde{y}))
    =
    B(A,x,y).
  \]
\end{rem}

We cannot directly apply this in our setup, since bounded derived categories of abelian categories do not admit all colimits. 
This forces us to pass to ind-categories:

\begin{lem}
\label{lem:ind-basics}
Let $\ccat$ be a small $\infty$-category.
Let 
\[
  j \mc \ccat \to \pcat(\ccat) 
  = 
  \funcat(\ccat^{\mathrm{op}},\mathrm{Ani})
\]
be the Yoneda embedding, and write $\ind(\ccat) \sse \pcat(\ccat)$ for the ind-category of $\ccat$.
\begin{enumerate}
\item 
\label{ind-basics:presentation}
A functor $F \in \pcat(\ccat)$ lies in $\ind(\ccat)$ if and only if there is a functor $p \mc \icat \to \ccat$ for a small filtered category $\icat$, such that $F$ is a colimit of the composition $j \circ p$.
\item
\label{ind-basics:lift-functor}
Let $f \mc \ccat \to \dcat$ be a functor between small categories.
Then $f$ lifts to a functor 
\[\ind(f)\mc \ind(\ccat) \to \ind(\dcat)\] that is compatible with the Yoneda embedding.
If $f$ is fully faithful, then $\ind(f)$ is fully faithful too.
\end{enumerate}
\end{lem}
\begin{proof}
Part \ref{ind-basics:presentation} is \cite[Cor.5.3.5.4]{htt}, part \ref{ind-basics:lift-functor} is \cite[Prop.5.3.5.10]{htt}.
\end{proof}

\begin{lem}
  \label{lem:lift-to-ind}
  \leavevmode
  \begin{enumerate}
  \item
  Let $\ccat$ be a small, idempotent-complete, stable, symmetric monoidal $\infty$-category.
  Then $\ind(\ccat)$ is a presentable, stable, symmetric monoidal $\infty$-category.
  \item
  Let $f \mc \ccat \to \dcat$ be a symmetric monoidal, exact functor between small, idempotent-complete, stable, symmetric monoidal $\infty$-categories.
  Then $\ind(f) $ is symmetric monoidal and colimit preserving.
  \end{enumerate}
\end{lem}
\begin{proof}
This is \cite[Lem.5.3.2.11]{higher-algebra}.
\end{proof}

\begin{lem}
  \label{lem:t-for-ind}
  Let $\ccat$ be a small, stable $\infty$-category with a t-structure.
  \begin{enumerate}
    \item
    \label{lem:t-for-ind:ant}
  There is a canonical t-structure on $\ind(\ccat)$ such that the embedding $\ccat \hookrightarrow \ind(\ccat)$ is t-exact and all cohomology functors commute with filtered colimits.
    \item
    \label{lem:t-for-ind:subcat}
    Let $\bcat \sse \ccat^{\heartsuit}$ be a full abelian subcategory.
    Set 
    \[
    \ccat_{\cohtext{\bcat}}
    = 
    \Set{x \in \ccat \given \text{$\hh^i(x) \in \bcat$ for all $i$}}
    \sse 
    \ccat,
    \]
    where we write $\hh^i(-)$ for the cohomology functors of the t-structure.
    Then 
    \[
      \ind(\ccat_{\cohtext{\bcat}})
      \sse 
      \Set{x \in \ind(\ccat) \given \text{$\hh^i(x) \in \ind(\bcat)$ for all $i$}}.
    \]
  \end{enumerate}
\end{lem}
\begin{proof}
Part \ref{lem:t-for-ind:ant} is \cite[Prop.2.13]{antieau}.
For \ref{lem:t-for-ind:subcat}, we use \cref{lem:ind-basics}: Let $x \in \ind(\ccat_{\cohtext{\bcat}}) \sse \ind(\ccat)$.
Write $x = \colim_{\alpha \in I} x_{\alpha}$ with $x_{\alpha} \in \ccat_{\cohtext{\bcat}}$ and $I$ filtered.
So by \ref{lem:t-for-ind:ant}, we have
\[
\hh^i(x)
=
\colim_{\alpha \in I} (\hh^i(x_{\alpha})) \in \ind(\bcat).
\qedhere
\]
\end{proof}

\section{Galois representations}

\subsection*{de Rham Galois representations}

\begin{numtext}
The inclusion $\repdr(G_K) \to \repqp(G_K)$ does not induce isomorphisms on $\ext^1$-groups.
In particular, the natural functor 
\[
\dd[b](\repdr(G_K)) \to \dd[b](\repqp(G_K))
\]
is not fully faithful, and containment in its essential image cannot be checked on the level of cohomology groups.
To obtain a lift of $R\Gamma(X_{\bark},\qp)$ along this functor, we use the description of $\dd[b](\repdr(G_K))$ given in \cite[Thm.4.12]{petrov-zavyalov}.
Before we state their theorem, let us recall the setup:
\end{numtext}

\begin{numtext}
Let $\dcat(G_K, \qp) = \dcat(\modcat(G_K,\qp))$ be the derived $\infty$-category of the category $\modcat(G_K,\qp)$ of $\qp$-modules on $BG_{K,\mathrm{pro\acute{e}t}}$.
This will be our appropriate ambient category.
Following \cite[\S 4]{petrov-zavyalov}, we consider the following categories and functors:
\begin{enumerate}
\item
By \cite[Lem.4.3]{petrov-zavyalov}, there is a fully faithful functor $\dd[b](\repqp(G_K)) \hookrightarrow \dcat(G_K,\qp)$.
\item
The topological ring $\bdrp$ yields a commutative algebra object in $\modcat(G_K,\qp)$.
We set 
\[
    \dcat(G_K,\bdrp) 
    =
    \modcat_{\bdrp}(\dcat(G_K,\qp)),
\]
i.e. the subcategory of $\bdrp$-modules in $\dcat(G_K,\qp)$.
Similarly, we define $\dcat(G_K,\bdr)$.
There are natural base-change functors 
\[
\dcat(G_K,\qp) \to \dcat(G_K,\bdrp) \to \dcat(G_K,\bdr).
\]
\item
For $\ccat$ an $\infty$-category, we write $\ccat_{\mathrm{fil}} = \fun((\zz,\leq)^\op,\ccat)$ for the associated category of filtered objects.
We set $\dfil(K) = \dcat(\vect(K))_{\mathrm{fil}}$, where $\dcat(\vect(K))$ is the derived $\infty$-category of $K$-vector spaces.
\end{enumerate}
\end{numtext}

  Theorem 4.12 of \cite{petrov-zavyalov} can now be stated as follows:
  \begin{theorem}
    \label{thm:pz-theorem}
    Let $K$ be a finite extension of $\qp$.
    Then $\dd[b](\repdr(G_K))$ is equivalent, as a symmetric monoidal $\infty$-category, to the fiber product of the following diagram of $\infty$-categories:
    \[
      \begin{tikzcd}
      \dd[b](\repqp(G_K))
      \ar{rd}[below left]{V \mapsto V \tensor_{\qp} \bdrp}
      &
      &
      \dfil(K)
      \ar{ld}
      \\ 
      &
      \dd(G_K,\bdrp)
      \end{tikzcd} 
    \]
  \end{theorem}
  Note that this equivalence is automatically exact, and that their proof constructs the equivalence as induced by the natural functors 
  \[
    \repdr(G_K) \hookrightarrow \repqp(G_K) \text{ and } \ddr \mc \repdr(G_K) \to \dfil(K).
  \]
\subsection*{semi-stable Galois representations}
We recall the following constructions, due to \cite{emerton-kisin} and \cite{deglise-niziol}:
\begin{construction}
Let $M \in \mfcatad(\pphi,N)$ be an admissible filtered $(\pphi,N)$-module.
Emerton-Kisin \cite{emerton-kisin} construct a complex $\kst(M) \in \chcat^{[0,2]}(\vect(\qp))$, which is given as
\[
    \kst(M)
    =
    \begin{tikzcd}[column sep = small]
        M
        \ar{r}[above]{d^0}
        &
        M \oplus M \oplus M_K / \fil^0 M_K 
        \ar{r}[above]{d^1}
        &
        M
    \end{tikzcd},
\]
with differentials 
\begin{equation}
\label{eq:ek-differential}
d^0 \mc x \mapsto (Nx, (\pphi-1)x, \bar{x})
\quad\text{and}\quad
d^1 
\mc 
(a,b,c)
\mapsto 
(1-p\pphi) a + N b.
\end{equation}
Let $\hom_{K_0}(M,T) \in \mfcatad(\pphi,N)$ be the internal Hom-object.
Then we have by \cite[Cor.2.4.4]{emerton-kisin} that
\begin{equation}
  \label{eq:ek-equation}
\hh^i(\kst(\hom_{K_0}(M,T)))
=
\ext^i_{\mfcatad(\pphi,N)}(M,T)
\end{equation}
holds for all $M,T \in \mfcatad(\pphi,N)$.
\par
To get a similar result for weakly admissible filtered $(\pphi,N,G_K)$-modules, D\'eglise-Nizio\l{} construct a complex $\hom^{\flat}(M,T) \in \chcat^{[0,2]}(\vect(\qp))$ such that 
\begin{equation}
  \label{eq:dn-equation}
\hh^i(\hom^{\flat}(M,T))
=
\ext^i_{\mfcatwa}(M,T)
\end{equation}
holds for all $i \geq 0$ \cite[Prop.2.7]{deglise-niziol}.
For our proof of \cref{prop:semistable-derived}, it will be enough to understand $\ext^2_{\mfcatwa}(M,T)$.
For this, we only need the homotopy quotient $\hom^{\sharp}(M,T)$ of $\hom^{\flat}(M,T)$, which is also constructed in \cite[\S 2]{deglise-niziol} and satisfies
\[
\hh^2(\hom^{\flat}(M,T)) = \hh^2(\hom^{\sharp}(M,T)).
\]
It can be constructed as follows:
For $M,T \in \mfcat(\pphi,N,G_K)$, we write $\hom_{G_K}(M,T)$ for the maps that are $\knr$-linear and equivariant for the ($\knr$-semilinear) $G_K$-action.
Moreover, let $\pphi_{\ast} \mc \vect(\knr) \isomorphism \vect(\knr)$ be the equivalence induced by the absolute Frobenius $\pphi \mc \knr \to \knr$.
Then $\hom^{\sharp}(M,T) \in \chcat^{[0,2]}(\vect(\qp))$ is defined as  
\[
\begin{tikzcd}[column sep = small]
\hom_{G_K}(M,T)
\ar{r}[above]{\delta^0}
&
\hom_{G_K}(M,\pphi_{\ast} T) 
\oplus
\hom_{G_K}(M,T)
\ar{r}[above]{\delta^1}
&
\hom_{G_K}(M,\pphi_{\ast}T)
\end{tikzcd},
\]
with differentials 
\begin{align*}
\delta^0 
\mc 
x 
&
\mapsto 
(\pphi_T x - \pphi_{\ast} x \pphi_M, N_T x - x N_M)
\\
\delta^1 
\mc 
(x,y)
&
\mapsto 
(\pphi_\ast N_T) x - p x N_M - p \pphi_T y + (\pphi_{\ast}y) \pphi_M.
\end{align*}
In particular, all $\ext^i_{\mfcatwa}$-groups vanish for $i \geq 3$, and we can use the complex $\hom^{\sharp}(M,T)$ to calculate $\ext^2_{\mfcatwa}$.
\end{construction}
Equipped with these tools, we can show:
\begin{prop}
  \label{prop:semistable-derived}
The inclusion $\repst(G_K) \hookrightarrow \repdr(G_K)$ extends to a fully faithful functor
\[
\dd[b](\repst(G_K))
\hookrightarrow
\dd[b](\repdr(G_K)).
\]
\end{prop}
\begin{proof}
    It suffices to see that the inclusion $\repst(G_K) \hookrightarrow \repdr(G_K)$ induces an equivalence on all Ext-groups.
    For $i = 0$, this is automatic, and for $i = 1$, this is known as Hyodo's theorem (cf. \cite[Thm.4.5.4]{emerton-kisin}).
    For the remaining $i$, we switch to the description of semi-stable (de Rham) representations by (weakly) admissible modules, via \cite{colmez-fontaine} and \cite{berger}.
    So it suffices to show that the inclusion
    \[
    j \mc \mfcatad(\pphi,N) \hookrightarrow \mfcatwa(\pphi,N,G_K),~ M \mapsto M \tensor_{K_0} \knr,
    \]
    with $G_K$ acting trivially on $M$ and canonically on $\knr$, induces an isomorphism on $\ext^i$-groups for $i \geq 2$.
    By \eqref{eq:ek-equation} and \eqref{eq:dn-equation}, it is clear that the only non-trivial case is $i=2$, where we can use $\hom^{\sharp}(M,T)$ for an explicit calculation:
    Observe first that for all $M,T \in \mfcatad(\pphi,N)$, we have
    \[
    \hom_{K_0}(M,T) = \hom_{G_K}(M\tensor_{K_0} \knr, T \tensor_{K_0} \knr),
    \]
    as $K_0$-vector spaces, since $(\knr)^{G_K} = K_0$.
    Moreover, note that the semilinear Frobenius $\pphi \mc M \to M$ becomes linear as a map $M \to  \pphi_\ast M$.
    In particular, we obtain the isomorphism:
    \[
      \begin{tikzcd}
      \alpha 
      \mc
      \hom_{K_0}(M,T)
      \ar{r}[above]{\pphi_{\ast}}
      &
      \hom_{K_0}(\pphi_{\ast} M , \pphi_{\ast} T)
      \ar{r}[above]{- \circ \pphi_{M}}
      &
      \hom_{K_0}(M,\pphi_{\ast} T)
      \end{tikzcd}
    \]
    \begin{claim*}
    The diagram 
    \[
    \begin{tikzcd}[ampersand replacement=\&, row sep = 2.5em, column sep = 4em]
      \hom_{K_0}(M,T)
      \oplus
      \hom_{K_0}(M,T)
      \ar{r}[above]{d' = ((1-p\pphi),N)}
      \ar{d}[left]{{\begin{pmatrix} 0 & \alpha \\ \id & 0 \end{pmatrix}}}
      \&
      \hom_{K_0}(M,T)
      \ar{d}[right]{\alpha}
      \\
      \hom_{K_0}(M,\pphi_{\ast}T)
      \oplus
      \hom_{K_0}(M,T)
      \ar{r}[below]{\delta^1}
      \&
      \hom_{K_0}(M,\pphi_{\ast}T)
    \end{tikzcd}
    \]
    commutes.
    \end{claim*}
    \begin{proof*}
    Let $f,g\in \hom_{K_0}(M,T)$.
    By definition, we have 
    \[
      \alpha(d'(0,g))
      = 
      \pphi_{\ast}N_T \circ \alpha(g)
      -
      \pphi_{\ast}g \circ \pphi_{\ast}N_M \circ \pphi_M
      =
      \pphi_\ast N_T \circ \alpha(g)
      - 
      p 
      \alpha(g)\circ N_M
      =
      \delta^1((\alpha(g),0)),
    \]
    where the second equality used the identity $\pphi_{\ast}N_M \circ \pphi_M = p \pphi_M \circ N_M$ as linear maps $M \to \pphi_\ast M$.
    Moreover, we have $\alpha(\pphi f) = \pphi_T \circ  f$, and thus 
    \[
    \alpha(d'(f,0))
    = 
    \pphi_\ast f \circ \pphi_M 
    -
    p 
    \pphi_T \circ  f 
    = 
    \delta^1((0,f)).
    \qedhere
    \]
    \end{proof*} 
    By the claim, we have $\dim_{\qp} \ext_{\mfcatad}^2(M,T) = \dim_{\qp} \ext_{\mfcatwa}^2(M,T)$.
    Moreover, the fact that $j$ induces an isomorphism on $\ext^1$ forces, by an elementary calculation \cite[Prop.3.3]{oort-calculates}, injectivity of the natural map 
    \[
    j_{\ast}
    \mc
    \ext_{\mfcatad}^2(M,T)
    \isomorphism
    \ext_{\mfcatwa}^2(M,T).
    \]
    Combining this with the dimension equality allows us to conclude.\footnote{The claim alone does not imply that $j_\ast$ is an isomorphism on $\ext^2$.}
\end{proof}

\begin{rem}
  \label{rem:crist-fails}
Note that an analogous statement for crystalline representations is not true --- the inclusion $\repcris(G_K) \hookrightarrow \repst(G_K)$ already fails to induce an isomorphism on $\ext^1$-groups:
Let $q \in \qptimes$ with $\abs{q} < 1$, and $E_q$ the associated Tate elliptic curve.
Its rational Tate module $V_p(E_q)$ sits in an exact sequence of the form 
\[
\begin{tikzcd}[column sep = small]
0
\ar{r}
&
\qp(1)
\ar{r}
&
V_p(E_q)
\ar{r}
&
\qp
\ar{r}
&
0
\end{tikzcd},
\]
cf. \cite[Rem.II.4.4]{berger-intro}.
Both $\qp(1)$ and $\qp$ are crystalline Galois representations, but $V_p(E_q)$ is semi-stable with non-trivial monodromy.
\end{rem}

\section{Relation between the bar construction and unipotent fundamental groups}

\begin{construction}
Let $f \mc X \to \spec(K)$ be an algebraic variety, and $x \in X(K)$ a $K$-rational point.
Then $x$ induces an augmentation
\[
R\Gamma(X_{\bark},\qp) \to \qp,
\]
by applying $Rf_\ast$ to the unit $\qp \to Rx_{\ast}x^\ast \qp$.
Using the notation from \ref{text:bar-const}, we set 
\begin{equation}
  \label{eq:bar-defn}
B(x) \defo (\qp)_{x^\ast} \tensor_{R\Gamma(X_{\bark},\qp)}^{L} {}_{x^\ast}(\qp) \in \ind \dd[b](\repqp(G_K)).
\end{equation}
\end{construction}

\begin{prop}[\cite{corrigendum-dancohen-schlank}]
    Let $F$ be an algebraically closed field in which $p$ is invertible, and let $X$ be an algebraic variety over $F$.
    Let $x\in X(F)$.
    Define $B(x)$ analogously to \eqref{eq:bar-defn}.
    Then there is a commutative diagram 
    \[
    \begin{tikzcd}[column sep = small]
    \locsysuni(X)
    \ar{rr}[above]{\sim}
    \ar{rd}
    &
    &
    \comod(\hh^0(B(x)))
    \ar{ld}
    \\
    &
    \vect(\qp)
    &
    \end{tikzcd}
    \]
    The horizontal map is an equivalence of categories, and the diagonal maps are induced by the augmentations.
    This diagram is moreover functorial in $X$.
\end{prop}
\begin{numtext}
In particular, this implies that there is a $G_K$-equivariant isomorphism
\[
  \hh^0(B(x))
  \cong 
  A_{\infty,x}
\]
for $B(x)$ associated  to an algebraic variety $X$ over $K$.
So if $B(x)$ lies in $\ind \dd[b](\repdr(G_K))$, this implies that $\pi_1^{\mathrm{\acute{e}t,uni}}(X_{\bark},x)$ is pro-de Rham in the sense of \cref{defn:de-rham-start}.
The analogous statement holds for \enquote{de Rham} replaced by \enquote{semi-stable} or \enquote{crystalline}.
\end{numtext}

\begin{numtext}
These results extend to path torsors:
Let $y \in X(K)$ be a second $K$-rational point.
Pullback along $x$ and $y$ induces fiber functors $x^{\ast},y^{\ast}\mc\locsysuni(X_{\bark}) \to \vect(\qp)$, respectively.
Consider the functor from $\qp$-algebras to sets given by 
\[
A \mapsto \Set{\text{$\tensor$-natural $A$-linear isomorphisms $A \tensor x^\ast \isomorphism A \tensor y^\ast$}}.
\]
This is pro-representable by an affine $\qp$-scheme
\[
\piuni(X_{\bark},x,y) = \spec(A_{\infty,x,y}),
\]
cf. \cite[Thm.9.48(a)]{milne-algebraic-groups}.
Moreover, $A_{\infty,x,y}$ is naturally an element of $\ind\repqp(G_K)$, cf. \cite[\S 2]{kim-selmer}.
We say that $\piuni(X_{\bark},x,y)$ is \emph{pro-de Rham} if $A_{\infty,x,y} \in \ind\repdr(G_K)$, and make the analogous definitions for semi-stability and crystallinity.
\par 
Let $y^\ast \mc R\Gamma(X_{\bark},\qp) \to \qp$ be the augmentation induced by $y$.
Again as in \ref{text:bar-const}, we can then consider 
\[
B(x,y) \defo (\qp)_{x^\ast} \tensor_{R\Gamma(X_{\bark},\qp)}^{L} {}_{y^\ast} (\qp)\in \ind \dd[b](\repqp(G_K)).
\]
\end{numtext}
\begin{prop}[\cite{corrigendum-dancohen-schlank}]
  There is a Galois-invariant isomorphism 
  \[
    \hh^0 (B(x,y))
    \isomorphism
    A_{\infty,x,y}.
  \]
\end{prop}

\section{Proof of de Rhamness and semi-stability}
\begin{theorem}
  \label{thm:main-theorem}
  Let $X$ be an algebraic variety over $K$, and let $x,y\in X(K)$.
  Then \[\hh^0(B(x,y))\in \ind(\repqp(G_K))\] is ind-de Rham.
\end{theorem}

\begin{proof}
  By \cref{lem:lift-to-ind}, the natural functor $\dd[b](\repdr(G_K)) \to \dd[b](\repqp(G_K))$ extends naturally to ind-categories, i.e. we have a commutative diagram of the form
  \[
  \begin{tikzcd}
  \dd[b](\repdr(G_K))
  \ar{d}[left]{\Psi}
  \ar[hookrightarrow]{r}
  &
  \ind \dd[b](\repdr(G_K))
  \ar{d}[right]{\ind \Psi}
  \\
  \dd[b](\repqp(G_K))
  \ar[hookrightarrow]{r}
  &
  \ind \dd[b](\repqp(G_K))
  \end{tikzcd}
  \]
  Theorem 4.12 of \cite{petrov-zavyalov} shows that we can lift $A = R\Gamma(X_{\bark},\qp) \in \calg(\dd[b](\repqp(G_K)))$ to an object 
  \[ \wtilde{A} \in \calg(\dd[b](\repdr(G_K))),\] in the case when $X$ is smooth proper and $\bark$ is replaced by $\cp$.
  But the proof of their theorem shows that the only non-formal input is that the \'etale-de Rham comparison for $\hh^i(X_{\bark},\qp)$ can be extended to a canonical isomorphism of algebra objects 
  \[
    R\Gamma(X_{\bark},\qp) 
    \tensor_{\qp}
    \bdrp
    \isomorphism 
    R\Gamma_{\mathrm{dR}}(X/K) 
    \tensor_K
    \bdrp
  \]
  in $\dd(G_K,\bdrp)$.
  For general algebraic varieties, this isomorphism is supplied by Beilinson's comparison isomorphism \cite[3.5.5]{beilinson}.
  In a similar way, we can lift the augmentations $A \to \qp$ to $\dd[b](\repdr(G_K))$, since they agree with the augmentations $R\Gamma_{\mathrm{dR}}(X/K) \to K$ of Deligne's de Rham cohomology after base-change to $\dd(G_K,\bdr)$ (here we use again that the augmentations come from rational points).
  Because $\ind \Psi$ is symmetric monoidal, it lifts to a symmetric monoidal functor $\ind(\Psi) \mc \modcat(\wtilde{A}) \to \modcat(A)$ by \cref{lem:induced-on-modules}.
  Arguing as in \cref{rem:lift-essential} shows that $B(x,y)$ lifts along $\ind(\Psi)$ to $\wtilde{B}(x,y) \in \ind(\dd[b](\repdr(G_K)))$.
  We conclude by \cref{lem:t-for-ind}.
\end{proof}

\begin{cor}
  \label{cor:main-cor}
  Keep the setup of \cref{thm:main-theorem}, and assume moreover that all cohomology groups $\hh^i(X_{\bark},\qp)$ are semi-stable.
  Then \[\hh^0(B(x,y)) \in \ind(\repqp(G_K))\] is ind-semi-stable.
\end{cor}

\begin{proof}
  By \cref{prop:semistable-derived}, we have a symmetric monoidal inclusion 
  \[\dd[b](\repst(G_K)) \hookrightarrow \dd[b](\repdr(G_K)).\]
  If all the $\hh^i(X_{\bark},\qp)$ are semi-stable, then we already have $R\Gamma(X_{\bark},\qp) \in \dd[b](\repst(G_K))$, since we can check this on cohomology groups by the previous sentence.
  So for the lift $\wtilde{B}(x,y)$ from the proof of \cref{thm:main-theorem}, we have in fact $\wtilde{B}(x,y) \in \ind \dd[b](\repst(G_K))$, as desired.
  The remaining argument is the same as in the proof of \cref{thm:main-theorem}.
\end{proof}

\section{Proof of crystallinity}
Despite \cref{rem:crist-fails}, we can also prove a slightly weaker crystalline analogue of \cref{thm:main-theorem} in the good reduction case.
For this, we use the more robust way to directly encode crystallinity of $R\Gamma(X_{\bark},\qp)$ via syntomification.
We do not recall the theory here, since we can run a formal argument once we know that the syntomic-to-\'etale realization has the following abstract properties:
\begin{theorem}
  \label{thm:syn}
Let $K/\qp$ be a finite extension.
Write $\repcriszp(G_K)$ for the category of finite-free $\zp$-modules $T$ with $G_K$-action such that $T[1/p]$ is crystalline.
\begin{enumerate}
\item
\label{thm:syn-realization}
There is a symmetric monoidal realization functor 
$
T_{\ok} 
\mc 
\perf(\oksyn)
\to 
\dd[b](\repcriszp(G_K)),
$
that becomes an equivalence after passing to isogeny categories:
\[
T_{\ok}[1/p]
\mc 
\perf(\oksyn)[1/p]
\isomorphism
\dd[b](\repcris(G_K)),
\]
\item
\label{thm:syn-cohomology}
 Let $f\mc \calx \to \spf(\ok)$ be a smooth proper formal scheme.
 Then $Rf_\ast \qp$ is contained in the essential image of $T_{\ok}[1/p]$.
\end{enumerate}
\end{theorem}
\begin{proof}
Part \ref{thm:syn-realization} is \cite[Prop.4.15]{pentland}.
Part \ref{thm:syn-cohomology} follows from \cite[Prop.4.7]{pentland}.
\end{proof}

\begin{theorem}
Let $K/\qp$ be a finite extension, and $X$ an algebraic variety that arises as the generic fiber of a smooth proper formal scheme $\calx$ over $\spf(\ok)$.
Let $x,y \in \calx(\ok)$.
Then \[\hh^0(B(x,y)) \in \ind(\repqp(G_K))\] is ind-crystalline.
\end{theorem}

\begin{proof}
Again by \cref{lem:lift-to-ind}, the functor $T_{\ok}[1/p]$ lifts to a symmetric monoidal equivalence 
\[
\ind(T_{\ok}[1/p])
\mc
\ind(\perf(\ok)[1/p])
\isomorphism 
\ind(\dd[b](\repcris(G_K))).
\]
Write $f\mc \calx \to \spf(\ok)$ for the structure map.
By \cref{thm:syn} we have $Rf_{\ast}\qp \in \im(T_{\ok}[1/p])$.
The essential image of $\dd[b](\repcris(G_K)) \to \dd[b](\repqp(G_K))$ is contained in 
\[
\dd^b_{\cohtext{\mathrm{cris}}}
=
\Set{ C \in \dd[b](\repqp(G_K)) \given \text{$\hh^i(C) \in \repcris(G_K)$ for all $i$}}
\sse 
\dd[b](\repqp(G_K)).
\]
Arguing as before shows that $B(x,y)$ lies in $\ind \dd^b_{\cohtext{\mathrm{cris}}}$.
\end{proof}

\printbibliography

  \par\bigskip
  {\small\itshape
    \noindent Email address:
    \href{mailto:wild@math.uni-frankfurt.de}{wild@math.uni-frankfurt.de}\par
  }
\end{document}